\documentclass{amsart}
\usepackage{latexsym, amssymb, amsmath, amsthm}
\usepackage[T1]{fontenc}
\usepackage{lmodern}
\usepackage{enumerate}
\usepackage{mathabx}
\usepackage{csquotes}
\usepackage[mathscr]{euscript}
\usepackage{parskip}
\usepackage{thmtools, thm-restate}
\usepackage{bussproofs}
\usepackage{xcolor}
\usepackage{caption}
\usepackage{float}

\usepackage{bm}

\newcommand{\xmark}{\ding{55}}%

\usepackage{pifont}

\declaretheorem[numberwithin=section]{theorem}
\newtheorem{lemma}[theorem]{Lemma}

\newtheorem{proposition}[theorem]{Proposition}

\newtheorem*{claim}{Claim}
\newtheorem*{theorem*}{Theorem}

\theoremstyle{definition}
\newtheorem{definition}[theorem]{Definition}
\theoremstyle{remark}
\newtheorem{remark}[theorem]{Remark}

\title{A classification of incompleteness statements}
\author{Henry Towsner}
\author{James Walsh}
\address{Department of Mathematics, University of Pennsylvania}
\email{htowsner@math.penn.edu}
\address{Department of Philosophy, New York University}
\email{jmw534@nyu.edu}
\thanks{Thanks to Hanul Jeon for discussion. Thanks to Mateusz Łełyk for discovering errors in an earlier draft. Thanks to an anonymous referee for helpful suggestions and, in particular, for suggesting the simple proof of Theorem \ref{main}, which we have included. The first author was partially supported by NSF Grant DMS-2054379.}

\begin{document}

\begin{abstract}
    For which choices of $X,Y,Z\in\{\Sigma^1_1,\Pi^1_1\}$ does no sufficiently strong $X$-sound and $Y$-definable extension theory prove its own $Z$-soundness? We give a complete answer, thereby delimiting the generalizations of G\"odel's second incompleteness theorem that hold within second-order arithmetic.
\end{abstract}

\maketitle

\section{Introduction}
G\"odel's second incompleteness theorem states that no sufficiently strong consistent and recursively axiomatized theory proves its own consistency. We give an equivalent restatement here:
\begin{theorem}[G\"odel]
    No sufficiently strong $\Pi^0_1$-sound and $\Sigma^0_1$-definable theory proves its own $\Pi^0_1$-soundness.
\end{theorem}
A theory is $\Pi^0_1$-sound (or, in general, $\Gamma$-sound) if all of its $\Pi^0_1$ theorems ($\Gamma$ theorems) are true. This notion can be formalized in the axiom systems we consider (see Definition \ref{reflection-dfn}).

A recent result \cite{walsh2021incompleteness} lifts G\"odel's theorem to the setting of second-order arithmetic, where stronger reflection principles are formalizable: 
\begin{theorem}[Walsh]\label{old-version}
No sufficiently strong $\Pi^1_1$-sound and $\Sigma^1_1$-definable theory proves its own $\Pi^1_1$-soundness.
\end{theorem}
Note that this latter theorem applies to all $\Sigma^1_1$-definable theories and not just to the narrower class of $\Sigma^0_1$-definable theories.

There are three classes of formulas in the statement of Theorem \ref{old-version}, leading to eight variations one could consider, including the original. In this paper we consider the other seven. The following table records the truth-values of the statement: \emph{No sufficiently strong $X$-sound and $Y$-definable theory proves its own $Z$-soundness.}

\begin{table}[H]
    \centering
    \begin{tabular}{cccc}
         & X  & Y & Z   \\
        \checkmark & $\Pi^1_1$  & $\Sigma^1_1$ & $\Pi^1_1$   \\
        \xmark & $\Pi^1_1$  & $\Pi^1_1$ & $\Pi^1_1$   \\
        \xmark & $\Sigma^1_1$  & $\Pi^1_1$ & $\Pi^1_1$   \\
        \xmark & $\Pi^1_1$  & $\Sigma^1_1$ & $\Sigma^1_1$ \\
        \xmark & $\Sigma^1_1$  & $\Sigma^1_1$ & $\Sigma^1_1$   \\
        \xmark & $\Pi^1_1$  & $\Pi^1_1$ & $\Sigma^1_1$   \\

        \xmark & $\Sigma^1_1$  & $\Sigma^1_1$ & $\Pi^1_1$   \\
        \checkmark & $\Sigma^1_1$  & $\Pi^1_1$ & $\Sigma^1_1$   \\
    \end{tabular}
    \label{tab:my_label}
\end{table}
To place the \xmark s on the table we show how to give appropriately non-standard definitions of arbitrarily strong sound theories. Theorem \ref{old-version} places the first \checkmark \ on the table; for this a ``sufficiently strong'' theory is any extension of $\Sigma^1_1\text{-}\mathsf{AC}_0$. For the second \checkmark \ a ``sufficiently strong'' theory is any extension of $\mathsf{ATR}_0$.

Both $\checkmark$s can be placed on the table via relatively simple reductions to G\"odel's original second incompleteness theorem. However, in \cite{walsh2021incompleteness}, it was emphasized that the first \checkmark (i.e., Theorem \ref{old-version}) can be established by a self-reference-free (indeed, diagonalization-free) proof, which is desirable since applications of self-reference are a source of opacity. In particular, the first \checkmark can be established by attending to the connection between $\Pi^1_1$-reflection and central concepts of ordinal analysis. To place the second \checkmark \ on the table we forge a connection between provable $\Sigma^1_1$-soundness and a kind of ``pseduo-ordinal analysis.'' Whereas $\Pi^1_1$-soundness provably follows from the well-foundedness of a theory's proof-theoretic ordinal, we show that $\Sigma^1_1$-soundness provably follows from the statement that a certain canonical ill-founded linear order lacks \emph{hyperarithmetic} descending sequences. In this way, we provide a proof with neither self-reference nor diagonalization of yet another analogue of G\"odel's second incompleteness theorem.

\section{The Proofs}

\subsection{Simplest Cases}
We begin by placing the first four \xmark s on the table.

\begin{definition}\label{reflection-dfn}
    When $\Gamma$ is a set of formulas, we write $\mathsf{RFN}_\Gamma(U)$ for the sentence stating the $\Gamma$-soundness of $U$ (i.e. reflection for formulas from $\Gamma$):
   $$\mathsf{RFN}_\Gamma(U):= \forall \varphi \in \Gamma\big( \mathsf{Pr}_U(\varphi) \to \mathsf{True}_\Gamma(\varphi) \big).$$
Here $\mathsf{True}_\Gamma$ is a $\Gamma$-definable truth-predicate for $\Gamma$-formulas. For the complexity classes that we consider this truth-predicate is available already in the system $\mathsf{ACA}_0$.
\end{definition}

For $\Gamma\in\{\Sigma^1_1,\Pi^1_1\}$, we let $\widehat\Gamma$ be the dual complexity class. The following result is an immediate consequence of this definition:
\begin{proposition}
    Provably in $\mathsf{ACA}_0$, for $\Gamma\in\{\Sigma^1_1,\Pi^1_1\}$, $T$ is $\Gamma$-sound if and only if $T+\varphi$ is consistent for every true $\widehat\Gamma$ sentence $\varphi$.
\end{proposition}

\begin{theorem}\label{general}
Let $\Gamma\in\{\Sigma^1_1,\Pi^1_1\}$. For any sound and arithmetically definable theory $S$, there is a sound and $\Gamma$-definable extension $T$ of $S$ such that $T\vdash \mathsf{RFN}_{\Gamma}(T)$.
\end{theorem}

\begin{proof}
   We define $U:=S+\Sigma^1_1\text{-}\mathsf{AC}_0$. Then we define:
   $$T(\varphi):= U(\varphi) \wedge \mathsf{RFN}_{\Gamma}(U)$$
   That is, $\varphi\in T$ if and only if both $\varphi\in U$ and $\mathsf{RFN}_{\Gamma}(U)$. 

   Then $\Sigma^1_1\text{-}\mathsf{AC}_0\vdash T=\emptyset \vee \big(T= U \wedge \mathsf{RFN}_{\Gamma}(U)\big)$. Thus, reasoning by cases, $\Sigma^1_1\text{-}\mathsf{AC}_0\vdash \mathsf{RFN}_{\Gamma}(T)$. Since $T= U \supseteq \Sigma^1_1\text{-}\mathsf{AC}_0$, $T\vdash \mathsf{RFN}_{\Gamma}(T)$.
   
   To see that $T$ is $\Gamma$-definable, note that $U$ is $\Gamma$-definable and that $\mathsf{RFN}_\Gamma(U)$
   has an arithmetic antecedent and a $\Gamma$ consequent.
   
   Finally, note that $T$ is just $U$, whence it is sound.
\end{proof}

\begin{remark}
    In the proof of Theorem \ref{general}, we use the $\Sigma^1_1$ choice principle only if $\Gamma=\Sigma^1_1$. Indeed, to infer that $\mathsf{RFN}_{\Sigma^1_1}(U)$ is $\Sigma^1_1$, we must pull the positively occurring existential set quantifier from $\mathsf{True}_\Gamma(\varphi)$ in front of a universal number quantifier. If $\Gamma=\Pi^1_1$, it suffices to define $U$ as $S+\mathsf{ACA}_0$, since $\mathsf{RFN}_{\Pi^1_1}$ has a finite axiomatization in $\mathsf{ACA}_0$.
\end{remark}

\subsection{Intermediate Cases}

We can resolve two more cases with a subtler version of the proof of Theorem \ref{general}. First, we recall the following useful lemma.

\begin{lemma}\label{reflection}
For $T$ extending $\mathsf{ACA}_0$, $\mathsf{RFN}_{\widehat\Gamma}(T)$ does not follow from any consistent extension of $T$ by $\Gamma$ formulas.
\end{lemma}
\begin{proof}
    Suppose $T+\gamma \vdash \mathsf{RFN}_{\widehat\Gamma}(T)$ with $\gamma\in\Gamma$. Then $T+\gamma \vdash \mathsf{Pr}_T(\neg \gamma)\to \neg \gamma$. Hence $T+\gamma \vdash \neg \mathsf{Pr}_T(\neg \gamma)$, i.e., $T+\gamma \vdash \mathsf{Con}(T+\gamma)$. So $T+\gamma\vdash \bot$.
\end{proof}

The following theorem adds two more \xmark s to our table.

\begin{theorem}\label{dual-positive-theorem}
Let $\Gamma\in\{\Sigma^1_1,\Pi^1_1\}$. For any sound and arithmetically definable theory $U$, there is a $\widehat\Gamma$-sound and $\widehat\Gamma$-definable extension of $U$ that proves its own $\Gamma$-soundness.
\end{theorem}
\begin{proof}
Consider the following formulas:
\begin{flalign*}
    \varphi(x) &:=   x=\ulcorner\mathsf{RFN}_{\Gamma}(U)\urcorner \vee x =\ulcorner \neg \mathsf{RFN}_{\widehat\Gamma}(U + \mathsf{RFN}_{\Gamma}(U)) \urcorner\\
    \tau(x) & := U(x) \vee \Big( \mathsf{RFN}_{\widehat\Gamma}\big(U + \mathsf{RFN}_{\Gamma}(U) \big) \wedge \varphi(x) \Big)
\end{flalign*}

Let $T$ be the theory defined by $\tau$.

\begin{claim}
    $T$ is $\widehat\Gamma$-definable via $\tau$.
\end{claim}
By inspection.

\begin{claim}
    $T$ is $\widehat\Gamma$-sound.
\end{claim}
Since $U$ is sound, $U+\mathsf{RFN}_\Gamma(U)$ is sound, so $\mathsf{RFN}_{\widehat\Gamma}(U+\mathsf{RFN}_\Gamma(U))$ holds, and therefore externally, we see that $T$ is the theory:
$$  U + \mathsf{RFN}_{\Gamma}(U) + \neg \mathsf{RFN}_{\widehat\Gamma}(U + \mathsf{RFN}_{\Gamma}(U)).$$
In particular, $T$ has the form $U'+\neg \mathsf{RFN}_{\widehat\Gamma}(U')$ where $U'$ is sound.
Suppose that $U'+\neg \mathsf{RFN}_{\widehat\Gamma}(U')\vdash \sigma$ where $\sigma$ is false $\widehat\Gamma$. Then $U'+ \neg \sigma \vdash \mathsf{RFN}_{\widehat\Gamma}(U')$. So $\mathsf{RFN}_{\widehat\Gamma}(U')$ follows from a consistent extension of $U'$ by $\Gamma$ formulas, contradicting Lemma \ref{reflection}.

\begin{claim}
    $T\vdash \mathsf{RFN}_{\Gamma}(\tau)$.
\end{claim}
From our external characterization of $T$ we see that 
$$T\vdash \neg \mathsf{RFN}_{\widehat\Gamma}(U + \mathsf{RFN}_{\Gamma}(U)).$$
Hence $T$ proves that $\tau$ defines the theory $U$. Again, appealing to our external characterization of $T$, $T\vdash \mathsf{RFN}_{\Gamma}(U)$. Thus, $T\vdash \mathsf{RFN}_{\Gamma}(\tau)$.
\end{proof}

\subsection{Limitations}

The presentation $\tau$ of theory $T$ defined in Theorem \ref{dual-positive-theorem} is clearly somewhat pathological, in part because $T$ cannot discern the identity of $\tau$. Before continuing to the final case, we want to illustrate that such pathologies are inevitable. We use a proof technique suggested at the end of \cite{walsh2021incompleteness}.

\begin{proposition}\label{limitation}
Let $T$ be a $\Gamma$-definable extension of $\Sigma^1_2\text{-}\mathsf{AC}_0$ that proves Theorem \ref{old-version} and Theorem \ref{main}. Suppose that there is a $\Gamma$ presentation $\tau$ of $T$ such that $T$ proves $\mathsf{RFN}_{\widehat\Gamma}(\tau)$. Then both of the following hold:
\begin{enumerate}
    \item There is a theorem $A$ of $T$ such that $T\vdash \neg \tau(A)$.
    \item There is a $\Gamma$ presentation $\tau^\star$ of $T$ such that $T$ proves $\neg \mathsf{RFN}_{\widehat\Gamma}(\tau^\star)$.
\end{enumerate}
\end{proposition}

\begin{proof}
Suppose that each of the following holds:
\begin{enumerate}
    \item $T$ is definable by a $\Gamma$ formula $\tau$;
    \item $T$ extends $\Sigma^1_2\text{-}\mathsf{AC}_0$;
    \item $T$ proves Theorem \ref{old-version} and Theorem \ref{main};
    \item $T$ proves the $\widehat\Gamma$-soundness of $\tau$.
\end{enumerate}
Let $\sigma$ be a sentence axiomatizing $\Sigma^1_2\text{-}\mathsf{AC}_0$. We have assumed $T\vdash \sigma$. We also have that $T\vdash \mathsf{RFN}_{\widehat\Gamma}(\tau)$. Let $A_1,\dots,A_n$ be the axioms of $T$ that are used in the $T$-proof of $\sigma\wedge \mathsf{RFN}_{\widehat\Gamma}(\tau)$. Thus:
$$\vdash (A_1\wedge\dots\wedge A_n) \to \big(\sigma\wedge \mathsf{RFN}_{\widehat\Gamma}(\tau)\big).$$

\begin{claim}
    $T\vdash \tau(A_1\wedge\dots\wedge A_n) \to \neg \mathsf{RFN}_{\widehat\Gamma}(\tau).$
\end{claim}
Reason in $T$. Suppose $\tau(A_1\wedge\dots\wedge A_n)$. Then $\tau$ extends $\Sigma^1_2\text{-}\mathsf{AC}_0$ and $\tau$ proves $\mathsf{RFN}_{\widehat\Gamma}(\tau)$. Since $\tau$ is a $\Gamma$ formula, Theorem \ref{old-version} (if $\Gamma=\Sigma^1_1$) or Theorem \ref{main} (if $\Gamma=\Pi^1_1$) entails that $\tau$ is not $\widehat\Gamma$-sound.

Since $T\vdash \mathsf{RFN}_{\widehat\Gamma}(\tau)$, the claim implies that $T\vdash \neg \tau(A_1\wedge\dots\wedge A_n)$.

On the other hand, consider $\tau^\star(x):=\tau(x) \vee x=\ulcorner A_1\wedge\dots\wedge A_n\urcorner$. Note that $\tau^\star$ is a $\Gamma$ definition of $T$. Yet we have just shown that $T\vdash \neg \mathsf{RFN}_{\widehat\Gamma}(\tau^\star)$.
\end{proof}

\begin{remark}
    Note that in the proof we need only assume that $T$ extends $\Sigma^1_2\text{-}\mathsf{AC}_0$ if $\Gamma=\Pi^1_1$. If $\Gamma=\Sigma^1_1$, it suffices to assume that $T$ extends $\Sigma^1_1\text{-}\mathsf{AC}_0$ since Theorem \ref{old-version} applies to extensions of $\Sigma^1_1\text{-}\mathsf{AC}_0$. Likewise, we need not assume that $T$ proves \emph{both} Theorem \ref{old-version} and Theorem \ref{main}. It suffices to assume that $T$ proves Theorem \ref{old-version} (if $\Gamma=\Sigma^1_1$) or that $T$ proves Theorem \ref{main} (if $\Gamma=\Pi^1_1$).
\end{remark}

\subsection{Hardest Case}

The only remaining case is the dual form of Theorem \ref{old-version}:
\begin{theorem}\label{main}
No $\Sigma^1_1$-sound and $\Pi^1_1$-definable extension of $\mathsf{ATR}_0$ proves its own $\Sigma^1_1$-soundness.
\end{theorem}

First we give a short proof that was discovered by an anonymous referee:
\begin{proof}
    Let $T$ be a $\Sigma^1_1$-sound and $\Pi^1_1$-definable extension of $\mathsf{ATR}_0$ that proves its own $\Sigma^1_1$-soundness. Let $\Phi$ be the (conjunction of) the finitely many statements used in the proof (assume that a single sentence axiomatizing $\mathsf{ATR}_0$ is among them). The sentence $\Phi\in T$ is true $\Pi^1_1$. Hence, $\Phi+\Phi\in T$ is consistent and $\Phi+\Phi\in T\vdash \mathsf{RFN}_{\Sigma^1_1}(T)$. By running this same argument inside $\Phi+\Phi\in T$, we conclude that $\Phi+\Phi\in T\vdash \mathsf{Con}(\Phi+\Phi\in T)$. Yet $\Phi + \Phi\in T$ is a consistent and finitely axiomatized extension of $\mathsf{ATR}_0$, which contradicts G\"odel's second incompleteness theorem.
\end{proof}
Note that a dual version of this proof also establishes Theorem \ref{old-version}.

For the rest of this section we will give an alternate proof. In \cite{walsh2021incompleteness}, Theorem \ref{old-version} was proved using concepts from ordinal analysis. In short, a connection is forged between $\Pi^1_1$-soundness and well-foundedness of proof-theoretic ordinals. Since we are now interested in $\Sigma^1_1$-soundness, we forge an analogous connection between $\Sigma^1_1$-soundness and \emph{pseudo-well-foundedness}, where an order is pseudo-well-founded if it lacks hyperarithmetic descending sequences.

For the rest of this section assume that $T$ is a $\Sigma^1_1$-sound and $\Pi^1_1$-definable extension of $\mathsf{ATR}_0$. In what follows, $\mathsf{PWF}(x)$ is a predicate stating that $x$ encodes a recursive pseudo-well-founded order (that is, a linear order with no hyperarithmetic decreasing sequence). A universal quantifier over \textsc{Hyp} can be transformed into an existential set quantifier in the theory $\mathsf{ATR}_0$. It follows that the statement $\mathsf{PWF}(x)$ is $T$-provably equivalent to a $\Sigma^1_1$ formula.


We will define $\prec_T$ to hold on pairs $(e,\alpha)$ where $e\in\textsc{Rec}$ and $\alpha\in dom(\prec_e)$.  We define $(e,\alpha)\prec_T(e',\beta)$ to hold if
\begin{quote}
there is some $f\in\textsc{Hyp}$ so that $\mathsf{Emb}(f,\prec_e\upharpoonright \alpha+1,\prec_{e'}\upharpoonright \beta)$ and $T \vdash \mathsf{PWF}(\prec_{e'})$.
\end{quote}
Here we write $\prec_e\upharpoonright \alpha+1$ for the restriction of the relation $\prec_e$ to $\{\gamma\in dom(\prec_e)\mid \gamma\preceq_e \alpha\}$.

To prove that $T\nvdash \mathsf{RFN}_{\Sigma^1_1}(T)$ it suffices to check that $T\vdash \mathsf{RFN}_{\Sigma^1_1}(T)\to \mathsf{PWF}(\prec_T)$ and that $T\nvdash \mathsf{PWF}(\prec_T)$. Let's take these one at a time.

\begin{claim}
    $T\vdash \mathsf{RFN}_{\Sigma^1_1}(T)\to \mathsf{PWF}(\prec_T)$.
\end{claim}

\begin{proof}
    Reason in $T$. Suppose $\neg \mathsf{PWF}(\prec_T)$. That is, there is a hyp descending sequence $f$ in $\prec_T$. Let $f(n)=(e_n,\beta_n)$. Thus we have:
    $$\forall n \; (e_{n+1},\beta_{n+1})\prec_T (e_n,\beta_n)$$
    By the definition of $\prec_T$, this is just to say:
    $$\forall n \; \exists g \in \textsc{Hyp} \; \mathsf{Emb}(g,f(n+1),f(n)).$$
    where we abuse notation to write $\mathsf{Emb}(g,f(n+1),f(n))$ for $\mathsf{Emb}(g,\prec_{e_{n+1}}\upharpoonright \beta_{n+1}+1,\prec_{e_{n}}\upharpoonright \beta_{n})$ to emphasize the role of $f$ in the statement.

    By the Kleene--Souslin Theorem \cite[Theorem VIII.3.19]{simpson2009subsystems} in $\mathsf{ACA}_0$, $f$ is $\Delta^1_1$ since $f$ is in $\textsc{Hyp}$, so the formula $\mathsf{Emb}(g,f(n+1),f(n))$ is equivalent to a $\Sigma^1_1$ formula (this is an application of $\Sigma^1_1\text{-}\mathsf{AC}_0$, which is a consequence of $\mathsf{ATR}_0$ \cite[Theorem V.8.3]{simpson2009subsystems}).

    $\mathsf{ATR}_0$ proves that $\textsc{Hyp}$ satisfies $\Sigma^1_1$ choice, and therefore proves
    $$\exists g\in\textsc{Hyp}  \; \forall n \; \mathsf{Emb}(g_n,f(n+1),f(n)).$$
    Note that $g$ is technically a set encoding the graphs of the countably many functions $g_n$ in the usual way.

    Using arithmetic comprehension, we form the composition $g_\star$ of the functions encoded in $g$---$g_\star(0)=g_0(\beta_1)$, $g_\star(1)=g_0(g_1(\beta_2))$ and so on. The function $g_\star$ is a hyp descending sequence in $\prec_{e_0}$, so $\prec_{e_0}$ is not pseudo-well-founded. Since $f(1)\prec_T f(0)$, we also have $T\vdash\mathsf{PWF}(\prec_{e_0})$. Recall that $\mathsf{PWF}(\prec_{e_0})$ is a $\Sigma^1_1$ claim. Hence, $\neg \mathsf{RFN}_{\Sigma^1_1}(T)$.
\end{proof}

Before addressing the second claim, let's record a dual form of Rathjen's formalized version of $\Sigma^1_1$ bounding \cite[Lemma 1.1]{rathjen1991role}.

\begin{lemma}\label{Rathjen}
    Suppose $H(x)$ is a $\Pi^1_1$ formula such that 
    $$\mathsf{ATR}_0 \vdash \forall x \big(H(x)\to \mathsf{PWF}(x)\big).$$ Then for some $e\in\textsc{Rec}$, $\mathsf{ATR}_0\vdash \mathsf{PWF}(e) \wedge \neg H(e)$.
\end{lemma}
\begin{proof}
    \cite[Theorem 1.3]{harrison1968recursive} implies that $\mathsf{PWF}$ (the set of pseduo-well-founded recursive linear orders) is $\Sigma^1_1$-complete (note that Harrison does not use self-reference or any other form of diagonalization in his proof). Hence, there is a total recursive function $\{k\}$ such that:
    $$\neg H(n) \Longleftrightarrow \mathsf{PWF}\big(\{{k}\}(n)\big).$$

    The reduction can be carried out in $\mathsf{ATR}_0$, so
    $$\mathsf{ATR}_0 \vdash \neg H(x) \leftrightarrow \mathsf{PWF}\big(\{{k}\}(x)\big).$$
    By the recursion theorem and the S-m-n theorem, there is an integer $e$ so that $\mathsf{ATR}_0$ proves that $\forall i[\{e\}(i)\simeq \{\{k\}(e)\}(i)]$ (where $\simeq$ means that if either side converges then both sides converge and are equal). Working in $\mathsf{ATR}_0$, $\neg \mathsf{PWF}(e)$ implies $\neg\mathsf{PWF}(\{k\}(e))$, which implies $H(e)$, which implies $\mathsf{PWF}(e)$, which is a contradiction. So $\mathsf{ATR}_0\vdash \mathsf{PWF}(e)$. (Not that this implies $e\in\textsc{Rec}$ by the definition of $\mathsf{PWF}(e)$.)

    Similarly, $H(e)$ implies $\neg \mathsf{PWF}(\{k\}(e))$, which is equivalent to $\neg\mathsf{PWF}(e)$, which we have already ruled out. So $\mathsf{ATR}_0\vdash \neg H(e)$.
\end{proof}

\begin{claim}
    $T\nvdash \mathsf{PWF}(\prec_T)$.
\end{claim}

\begin{proof}
Suppose that $T$ proves $\mathsf{PWF}(\prec_T)$. From the definition of $\prec_T$, it follows that:
$$T\vdash \big(\exists f\in \textsc{Hyp} \; \mathsf{Emb}(f,\prec_x,\prec_T)\big) \to \mathsf{PWF}(\prec_x).$$
The formula $\exists f\in \textsc{Hyp} \; \mathsf{Emb}(f,\prec_x,\prec_T)$ consists of an existential hyp quantifier before a $\Pi^1_1$ matrix (the matrix is $\Pi^1_1$ since $\prec_T$ refers to provability in $T$ and $T$ is $\Pi^1_1$-definable). Hence, there exists a $\Pi^1_1$ formula $\pi(x)$ such that:
$$\mathsf{ATR}_0\vdash \pi(x) \leftrightarrow \exists f\in \textsc{Hyp} \; \mathsf{Emb}(f,\prec_x,\prec_T).$$
By Lemma \ref{Rathjen}, there is some $e$ so that 
$$\mathsf{ATR}_0 \vdash \mathsf{PWF}(\prec_e) \wedge \neg\pi(e).$$
Hence $\mathsf{ATR}_0\vdash \neg \exists f\in \textsc{Hyp} \; \mathsf{Emb}(f,\prec_e,\prec_T)$. Moreover, since $\mathsf{ATR}_0$ is sound, we infer that $\neg \exists f\in \textsc{Hyp} \; \mathsf{Emb}(f,\prec_e,\prec_T)$ is true.

On the other hand, since $T$ extends $\mathsf{ATR}_0$, we infer that $T\vdash\mathsf{PWF}(\prec_e)$. Hence the map $\alpha\mapsto (e,\alpha)$ is a canonical hyp embedding of $\prec_e$ into $\prec_T$. So $\neg \exists f\in \textsc{Hyp} \; \mathsf{Emb}(f,\prec_e,\prec_T)$ is false after all. Contradiction.
\end{proof}

It follows from the claims that $T\nvdash \mathsf{RFN}_{\Sigma^1_1}(T)$, which completes the proof of Theorem \ref{main}.
\bibliographystyle{plain}
\bibliography{bibliography}

\end{document}